\documentclass[12pt,oneside]{article}
\usepackage{Paper_style} 
\newcommand{\Para}{\mathcal{P}}
\newcommand{\Quasi}{\mathcal{Q}}

\newcommand{\Family}{\mathcal{F}}

\newcommand{\Pow}{3}
\newcommand{\Torus}{\mathbb{T}^2}
\newcommand{\wind}{\mathcal{W}}
\newcommand{\SkewP}{\mathcal{S}}
\newcommand{\Dioph}{\mathcal{D}}
\newcommand{\Brun}{\mathcal{B}}

\begin{document}
\title{\vspace{-1cm}Almost-sure quasiperiodicity in countably many co-existing circles.\vspace{0cm}}
\author{Suddhasattwa Das\footnotemark[1]}
\footnotetext[1]{Courant Institute of Mathematical Sciences, New York University}
\maketitle

\begin{abstract} In many dynamical systems, countably infinitely many invariant tori co-exist. The occurrence of quasiperiodicity on any one of these tori is sometimes sufficient to establish strong global properties, like dense trajectories and periodic points. In this paper, we establish sufficient conditions for a countably infinite collection of parameterized circle diffeomorphisms to have quasiperiodic behavior on at least one of the circles, for a full Lebesgue measure set of the parameter values. As an application, we study parameterized families of skew-product maps on the torus and prove sufficient conditions for the existence of at least on quasiperiodic circle for Lebesgue-almost every parameter value. 
\end{abstract}

\section{Introduction and main results}

In this paper, the unit circle $S^1$ will be identified as $\real /\num $, and $proj:\real \to\num $ is the associated quotient map. A homeomorphism of the circle $f:S^1\to S^1$ can be lifted to a map $\bar{f}:\real \to\real $ under the covering map $proj$. It is well known (see for example, \cite{HermanSeminal}) that the following limit exists and is a constant independent of $z$. 
\begin{equation}\label{eqn:rot_num_defn}
\rho(f):=\underset{n\to\infty}{\lim}\frac{\bar{f}(z)-z}{n}
\end{equation}
This limit is called the \textbf{rotation number} of $f$. The rotation number is of fundamental importance in inferring the properties of the map and its limit points. If the rotation number is rational of the form $\frac{p}{q}$, then all points on $S^1$ is in the basin of attraction of some $q$-periodic point. On the other hand, if $\rho\notin\mathbb{Q}$, then $f$ has the rotation $\theta\mapsto\theta+\rho\bmod 1$ as a factor map. In fact, a homeomorphism of the circle is conjugate to an irrational rotation iff it is transitive (see \cite{TransCircleHomeo}). Dynamics on a torus $\mathbb{T}^d$ or on $S^1$ which are conjugate to an irrational rotation are called ``quasiperiodic dynamics". They have interesting properties like transitivity, non-mixing, unique ergodicity, zero Lyapunov exponents etc. This paper establishes sufficient conditions under which a countably infinite family of invariant circles will have quasiperiodicity on at least one circle, for Lebesgue almost-every value of a parameter $t$. 

The existence of a single invariant curve/manifold, on which the dynamics is transitive, can often lead to strong global properties. See for example blenders in \cite{Blender1, Blender2, Blender3}, or \cite{DasJim1} discussed in more details later. The main theorem will be applied to obtain sufficient conditions for the existence of an invariant quasiperiodic curve in a skew-product map on the torus, for Lebesgue-a.e. value of a parameter.

An orientation preserving circle homeomorphism $F$ is of the form given below.
\begin{equation}\label{eqn:lift}
F(z)=z+g(z) \bmod 1
\end{equation}
where $g:\real \to\real $ is called the \textbf{periodic part} of the map $F$, and is periodic and of the same smoothness class as $F$. We are interested in parameterized families of $C^\Pow$ circle diffeomorphisms, parameterized by a parameter $t\in [0,1]$, which can be written similar to Eq. \myref{eqn:lift} in the following manner.
\begin{equation}\label{eqn:lift_family}
f_t:\theta\mapsto\theta+Nt+g_t(\theta)\bmod 1
\end{equation}
where $N\in\mathbb{N}$ is called the \textbf{t-winding number} of $f$ and is denoted as \boldmath$\wind(f)$\unboldmath. The choice of the $t$-winding number is not unique, but we require that it must satisfy $|\partial_t g_t|<1$.

\textbf{Partition of the parameter space.} Given a parameterized family $f_t$, define \boldmath $\Para(f_t)$ \unboldmath the set \{$t\in [0,1]$ : $f_t$ has a periodic point\}, and by \boldmath $\Quasi(f_t)$ \unboldmath the set \{$t\in [0,1]$ : $f_t$ is topologically conjugate to an irrational rotation\}. 
By Denjoy's theorem \cite{DenjoyC2}, a $C^3$ circle diffeomorphism with an irrational rotation number $\rho$ is topologically conjugate to the rotation $\theta\mapsto\theta+\rho\bmod 1$. Therefore,
\[[0,1]=\Para(f_t)\sqcup \Quasi(f_t).\]

\textbf{Arnold tongues and periodic windows.} If for some parameter value $t_0$, $f_{t_0}$ has a stable periodic orbit with period $n$, then $\rho(f_t)$ is constant and $=k/n$ for $t$ in some neighborhood of $t_0$. These intervals over which $\rho(f_t)$ is constant are called periodic windows. The dependence of the rotation number on the parameter $t$ has been studied for over 50 years, see for example \cite{HermanSeminal, Arnold} etc. Arnold, in the seminal paper \cite{Arnold} studied the family
\[f_{t,\delta}:\theta\to\theta+t+\delta\sin(2\pi \theta)\bmod 1\]
and proved that $\mu(\Para(f_t))\to 0$ as $\delta\to 0$ and $\mu(\Para(f_t))\to 1$ as $\delta\to 1/2\pi$. In his example, each of the countably infinitely many periodic windows shrink in width to a point at $\delta=0$, as $\delta\to 0$, and monotonically thicken as $\delta$ is increased. The bifurcation diagrams of these windows with $\delta$ are called ``Arnold tongues'' because of their shape. The scaling laws of their width with parameter $\delta$ has been studied extensively in the general setting of $f_t\in\Family$, in which $\sin(2\pi \theta)$ is replaced by some general periodic function $g_t(\theta)$. Brunovsky  proved that there is a set of 1-parameter circle diffeomorphisms $\Brun$ which is residual in the family of $C^3$-circle diffeomorphisms, such that for $\forall f_t\in\Brun$, the periodic windows contain an open, dense set in the parameter space $[0,1]$. See Proposition 3 from \cite{Brunovsky3}

\textbf{Universality of Arnold tongues.} Many universal properties of Arnold tongues have been observed in general families of the form Eq. \myref{eqn:lift_family}. Cvitanovic et. al. \cite{ScalingTongue1} ordered the tongues based on the ``Farey-sequence'' ordering of the rationals and for all fixed values of $t$, found asymptotic scaling laws wrt their number in this ordering. Jonker \cite{ScalingTongue2} proved a $q^{-3}$ scaling law, where $q$ is the denominator of the rotation number. The differentiability properties of the boundaries of the Arnold tongues and their angle of contact at $\delta=0$ has been studied in \cite{Kuntal1, Kuntal2}. Our main theorem (Theorem \myref{thm:A}) says that if a countably infinite family of parameterized circle diffeomorphisms satisfy certain conditions, then for Lebesgue-a.e. $t\in[0,1]$, the intersection of the Arnold tongues has measure zero, regardless of their scaling laws and thickness. In Corollary \myref{corr:B}, we use this theorem to prove the existence of quasiperiodic circles in an open family of torus maps. 

Define \boldmath $\|f_t\|_{\Family}$ \unboldmath :=$\max\left( \|g_t\|_{C^\Pow}, \left|\frac{\partial}{\partial t} g_t\right|_{C^0}\right)$. Note that $\|\frac{\partial}{\partial t} g_t\|_{C^0}$=$\max$\{ $\left|\frac{\partial}{\partial t} g_t(\theta)\right|$ : $t\in[0,1]$, $\theta\in S^1$ \}. $\|g_t\|_{C^3}$ denotes the $C^\Pow$ norm of $g_t$ as a function of $\theta$.

Let \boldmath $\Family$ \unboldmath denote the following set of parameterized circle diffeomorphisms. 
\[\Family:=\{f_t\in\Brun \mbox{ for } t\in [0,1]\ :\ \sup_{t,\theta} \left|\|f_t\|_{\Family}\right| < 1 \}.\]

We will denote by \boldmath $\mu$ \unboldmath the 1-dimensional Lebesgue measure. 

\begin{theorem}[Main theorem]\label{thm:A}
Let $(f_{n,t})_{n\in\mathbb{N}}$ be a sequence in $\Family$ and written in the form Eq. \myref{eqn:lift_family}. Suppose that the following two conditions are satisfied.
\\(A1) $\underset{n\to\infty}{\lim\sup}\wind(f_{n,t})=\infty$.
\\(A2) $\OpA{\sup}{n\in\num}\|f_{n,t}\|_{\Family}<1$.
\\Let $\Para:=\underset{n\in\mathbb{N}}{\cap}\Para(f_{n,t})$. Then if $Int(\Para)=\emptyset$, then $\Para$ has zero Lebesgue measure.
\end{theorem}

\textbf{Remark.} In other words, given a countably infinite family of parameterized circle diffeomorphisms satisfying (A1)-(A2), there is a full measure set of $t\in I$ for which one of the following two holds.
\\(i) $\exists n\in\mathbb{N}$ for which $f_{n,t}$ is conjugate to an irrational rotation.
\\(ii) $t$ belongs to a common periodic window for all $f_{n,t}$-s. Generically, this occurs only if there is a curve $\Gamma$ transverse to all the vertical circles and invariant under $F^p$ for some $p\in\num$.

\textbf{Remark.} Condition (A2) is equivalent to saying that $\exists r\in (0,1)$ such that for $\forall n\in\num$, $\forall t\in [0,1]$, $\theta\in S^1$, $\|g_{n,t}\|_{C^3}<r$, $\left|\frac{\partial}{\partial t} g_{,t}(\theta)\right|$=$\left|\frac{\partial}{\partial t} f_{n,t}(\theta)/\wind(f_{n,t})-1\right|$ $<r$.

\textbf{Skew product maps.} Countable families of parameterized circle diffeomorphisms arise naturally in skew product maps. Let \boldmath $\SkewP$ \unboldmath denote the following family of skew-product maps on the torus.
\begin{equation}\label{eqn:map}
F_t:(x,y) = (mx,\ y+t+g_t(x,y))\ \mbox{mod 1 in each variable}.
\end{equation}
where $m>1$ is an integer, $t\in[0,1]$ is a parameter and $g_t:\Torus\to S^1$ is $C^0$, $C^\Pow$ with respect to $x$, $y$ respectively. For $\forall x_0\in S^1$, a \textbf{vertical circle}, \boldmath $S_{x_0}$ \unboldmath, is the set $\{(x_0,y)\ |\ y\in S^1\}$. Then vertical circles are mapped into vertical circles, i.e., for every $x\in S^1,\ F(S_x)=S_{F(x)}$. For each positive integer $p$, $x_0$ is fixed under the $p$-th iteration of the circle map $x\mapsto mx\ (\bmod\ 1)$ iff
\begin{equation}\label{eqn:periodic_x_coord}
x_0=\frac{k}{m^p-1}\ (\bmod\ 1)
\end{equation}
where $k$ is an integer. Each vertical circle $S_{x_0}$ with ${x_0}$ of the form (\myref{eqn:periodic_x_coord}) will be called a
\textbf{periodic circle} of the map. These circles will also be denoted as \boldmath $S_{k,p}$ \unboldmath. The torus map can be restricted to these circles to get a countable number of orientation preserving circle homeomorphisms.
\begin{equation}\label{eqn:restricted_map}
F_{k,n}:=F^n|S_{k,n}:S_{k,n}\to S_{k,n}
\end{equation}

Das et. al. \cite{DasJim1} proved that if at least one of the $F_{k,n}$-s is conjugate to an irrational rotation, then the map $F$ has strong mixing, has dense repellors, saddles and dense unstable manifolds of its saddles. In this paper we will prove that this hypothesis is satisfied by a family of torus maps $F_t:\Torus\to\Torus$ from Eq. \myref{eqn:map}, satisfying the following condition. for some $R>0$.

\boldmath $(A3)_R$ \unboldmath : There are infinitely many $ k,n\in\num $, $0<k<m^n$ for which $\|F_{k,n}|S_{k,n}-Id\|_{C^3y}<R$. \ ($F_{k,n}$, $S_{k,n}$ is as in Eq. \ref{eqn:restricted_map}.)

\begin{corollary}[Quasiperiodicity on skew product maps]\label{corr:B}
Let $R\in(0,1)$. Then there is a residual family of skew product maps of the form Eq. \myref{eqn:map}  satisfying $\left|\partial_t g_t\right|<R$ such that if $(A3)_R$ holds, then for Lebesgue-a.e. $t\in [0,1]$, either one of the following two conclusions hold.
\\(i) there is a circle $S_{k,n}$ on which the map $F_{k,n}$ from (\myref{eqn:restricted_map}) is quasiperiodic; OR
\\(ii) $F_t$ divides the rectangle into invariant cylinders (topologically $S^1\times[0,1]$ with at least one attracting boundary component.
\end{corollary}
\begin{proof} We will only consider the $(k,n)$-s which satisfy $(A3)_R$. Let $x_0$ be of the form in Eq. \myref{eqn:periodic_x_coord} for some $k,n\in\mathbb{N}$. Then for $\forall y\in S^1$, $F^n_t(x_0,y)=(x_0,y+nt+\sum\limits_{i=0}^{n-1}g_t(z_{i}))$, where $z_i=F_t^i(\frac{k}{m^n-1},y)$. Thus the associated circle map is 
\[F_{(k,n),t}:\theta\mapsto\theta+nt+\sum\limits_{i=0}^{n-1}g_t(z_{i})\bmod 1.\]
Note that the $F_{(k,n),t}$-s form a countable collection of parameterized circle diffeomorphisms of the form (\myref{eqn:lift_family}). 

It will now be shown that assumptions (A1) and (A2) of Theorem \myref{thm:A} are satisfied by $F_{(k,n),t}$.

Note that $\wind(F_{(k,n),t})=n$, which $\to\infty$ as $n\to\infty$. Therefore, (A1) is satisfied.

Note that $\frac{\partial_t F_{(k,n),t}}{\wind(F_{(k,n),t})} =1+\frac{1}{n}\sum\limits_{i=0}^{n-1}\partial_t g_t(z_{i})$. 
\\Therefore, $\left|\frac{\partial_t F_{(k,n),t}}{\wind(F_{(k,n),t})}-1\right| \leq \frac{1}{n}\sum\limits_{i=0}^{n-1}|\partial_y g_t(z_{i})|\leq R$.

Moreover, if $F_{(k,n),t}(\theta)=\theta+nt+g_{(k,n),t}(\theta)$, then $\|F_{(k,n),t}\|_{\Family}$ = $\|g_{(k,n),t}\|_{C^3}$ = $\|F_{k,n}|S_{k,n}-Id\|_{C^3y}<R$. Therefore, (A2) is satisfied and the conclusion of Theorem \ref{thm:A} holds. Suppose the parameter $t$ lies in the interior of a common periodic window for all the circles $S_{k,n}$. For a $C^3$-residual family of maps, this implies that these circles have periodic saddles $x_{k,n}$, which behave as attracting fixed points for the circle maps $F_{k,n}$. Again, for a residual set of maps, this is only possible if these saddles lie on a countable collection of invariant curves, which are common unstable manifolds for these saddles. These curves must be closed and partition the torus into cylinders as claimed. \qed


\end{proof}

\textbf{Remark.} Kleptsyn and Nalskii \cite{StepSkew3} proved that under generic conditions, a.e. orbit of smooth maps in $\SkewP$ converge. In particular, this implies that for a.e. fibre $S_x$, the orbits of $(x,y_1)$, $(x,y_2)$ converge under the action of $F$ for a.e. $y_1,y_2\in S^1$. 
Corollary \myref{corr:B} therefore makes the non-trivial statement, that on a $0$-measure set of fibres, namely, the periodic circles in Eq. \myref{eqn:periodic_x_coord}, the map will act as a rotation for a.e. $t\in[0,1]$. Homburg \cite{Homburg} constructed families of skew products on the torus which are robustly topologically mixing and at the same time exhibits this fibre-wise contraction property.

\textbf{Notation.} For $\forall k\in\mathbb{N}$, \boldmath $[k]$ \unboldmath denotes the set \{$1,\ldots,k$\}. Let $f_t:[0,1]\times S^1\to S^1\in\Family$, \boldmath $\|f_t\|_{\Family}$ \unboldmath $:=\sup(\|\partial_t g_{k,t}\|_{C^0}, \|\partial_\theta g_{k,t}\|_{C^{2}}\}$. The abbreviation \textbf{WLOG} means ``without loss of generality''.
\section{Proof of Theorem \ref{thm:A}.}

We begin the proof by a simple consequence of the assumption that $Int(\Para)=\emptyset$.
\begin{lemma}\label{lem:Para_empty_int}
If $Int(\Para)=\emptyset$, then for every interval $ J\subset [0,1]$, there exists infinitely many $n$-s for which $\rho(f_{n,t})$ is not constant for $t\in J$.
\end{lemma}
\begin{proof} By the assumption that $Int(\Para)=\emptyset$, there is at least one $n$ such that $\rho(f_{n,t})$ is not constant for $t\in J$. Note that for $\forall n$, $\Para_n=\Para(f_{n,t})$ contains an open dense set. Suppose that there are only finitely many $n$-s $n_1,\ldots,n_k$ for which $\rho(f_{n_k,t})$ is not constant for $t\in J$. Therefore, for all other $n$-s, $\rho(f_{n,t})$ is constant for $t\in J$. In particular, for all but finitely many $n$-s, $J\subset \Para_n$. Let $J'=\underset{i\in[k]}{\cap}P_{n_i}$. Then $J'$ is an open dense set and
$J'\cap J$ contains an open interval $J''$. Note that by assumption, $J''\subset \underset{n\in\mathbb{N}}{\cap}P_{n}=\Para$. This is a contradiction of the fact that $Int(\Para)=\emptyset$. Therefore, our assumption of finiteness of $n$-s was wrong. \qed
\end{proof}
The heart of the proof of Theorem \myref{thm:A} is the following Proposition, which will be proved in Section \myref{sect:uni_bound}.
\begin{proposition}\label{lem:unif_bound_1}
For $\forall R\in[0,1)$, $\exists\eta\in(0,1)$ such that if a parameterized family $f_t\in\Family$ satisfies $\left|f_t\right|_{\Family}\leq R$, then $\mu(\Para(f_t))<\eta$.
\end{proposition}
Proposition \myref{lem:unif_bound_1} will be used to establish a universal bound described below.

\begin{lemma}\label{lem:renormalization}
There exists \boldmath $\eta$ \unboldmath $\in(0,1)$ such that for any subinterval $J\subset I$, $\exists n\in\mathbb{N}$ such that $\mu(\Para(f_{n,t})\cap J)<\eta\mu(J)$.
\end{lemma}
\begin{proof} Let $J=(a,b)\subset [0,1]$ be fixed for the rest of the proof. By Lemma \myref{lem:Para_empty_int}, there are arbitrarily large $n\in\mathbb{N}$ for which $\rho(f_{n,t})$ is not constant for $t\in J$. Let $f_{n,t}(\theta)=\theta+nt+g_{n,t}(\theta)$. By assumption (A1), it can be assumed WLOG that $n=\wind(f_{n,t})$ is large enough so that it satisfies $\left|n(b-a)-k\right|<0.5$ for some $k\in\mathbb{N}$. Let $\hat{f}$ be the parameterized family $f_{n,t}$ with $t$-renormalized over $[a,b]$, as shown below.
\[\hat{f}_t(\theta):=\theta+n(a+(b-a)t)+g_{n,a+(b-a)t}(\theta),\ t\in [0,1].\]
Note that $|\partial_\theta \hat{f}-1|=|\partial_\theta {f_{n,t}}-1|$ which is $<R$ by (A2) and $\wind(\hat f)=k$ and $\partial_t\hat{f}=N(b-a)+(b-a)\partial_t g_{n,t}$. Therefore
\[\frac{\partial_t \hat f}{\wind \hat f} -1 = \frac{N(b-a)}{k}+\frac{b-a}{k}\partial_t g_{n,t}-1<k^{-1}\left(0.5+|\partial_t g_{n,t}|\right).\]
Since $k$ can be chosen to be large enough, the RHS of the above equation can be made $<R$. Therefore, $\mu(\Para(\hat f))<\eta$ by Proposition \myref{lem:unif_bound_1}, where the constant $\eta$ depends only on $R$. Therefore, $\mu(\Para(f_{n,t})\cap J)=\mu(J)\mu(\Para(\hat{f}))<\eta\mu(J)$. \qed
\end{proof}

For $\forall n\in\mathbb{N}$, let \boldmath $\Para_n$\unboldmath $:=\Para(f_{n,t})$. One more lemma is needed to complete the proof of Theorem \myref{thm:A}.
\begin{lemma}[Renormalization lemma]\label{lem:renormalization_2}
For $\forall$ open subset $U\subseteq [0,1]$, $\forall\epsilon'>0$, $\exists$ a subset $V\subset U$ which is a finite union of open intervals, and a finite set $M\subset \mathbb{N}$ such that $\mu(U-V)<\epsilon'$ and 
\[V \subset \left[\underset{i\in M}{\cap}\Para_{n_i}\right]\cap U \mbox{ and } \mu(\Para\cap V)<\eta\mu(U).\]
\end{lemma}
\begin{proof}
Note that there is a finite, disjoint collection of intervals $I_1,\ldots,I_k$ which form connected components of $\Para_1\cap U$ such that if $V=\underset{i\in[k]}{\sqcup}I_i$, then $\mu(U-V)<\epsilon'$. Now for every $i\in[k]$, by Lemma \myref{lem:renormalization}, $\exists n_i\in\mathbb{N}$ such that $\mu(\Para_{n_i}\cap I_i)<\eta\mu(I_i)$. Set $M:=\{n_1,\ldots,n_k\}$. Therefore,
\[\Para\cap V \subset \left[\underset{i\in M}{\cap}\Para_{n_i}\right]\cap V   = \left[\underset{i\in M}{\cap}\Para_{n_i}\right]\cap \left[\underset{j\in [k]}{\cup}I_j\right] \subset \underset{i\in[k]}{\sqcup}\Para_{n_i}\cap I_i.\]
By Lemma \myref{lem:renormalization}, $\mu(\Para_{n_i}\cap I_i)<\eta \mu(I_i)$, so $\mu\left(\Para\cap V\right)<\eta\OpA{\Sigma}{i\in [k]}\mu I_i<\eta\mu(U)$. \qed
\end{proof}

\textbf{Proof of the theorem.} Let \boldmath $\epsilon>0$ \unboldmath be arbitrarily chosen. Then the theorem will be proved if it can be shown that $\mu(\Para)<2\epsilon$. 
We will now inductively define a sequence of sets $J_k\subset I$, each of which is an approximation of $\Para$ and is a finite union of open intervals. Let $J_{-1}=J_0=[0,1]$, $M_0$ be the set \{$1$\}, and assume as an inductive step that a finite collection of open intervals $J_{k-1}\subset I$, and a finite $M_{k-1}\subset\mathbb{N}$ has been defined for some $k\in\mathbb{N}$ which satisfy $J_{k-1} \subset \left[\underset{i\in M_{k-1}}{\cap}\Para_{n_i}\right]\cap J_{k-2}$. Now use $J_{k-1}$ as the set $U$ in Lemma \myref{lem:renormalization_2} and choose $J_k$ to be the set $V$, with $\epsilon'=2^{-k}\epsilon$. Therefore, the following are satisfied.
\\(i) $J_K \subset \left[\underset{i\in M_k}{\cap}\Para_{n_i}\right]\cap J_{k-1}$.
\\(ii) $\mu\left(\left[\underset{i\in M_k}{\cap}\Para_{n_i}\right]\cap J_{k}\right) = \mu(\Para\cap J_k)<\eta \mu(J_{k-1})$.
\\(iii) if $E_{k-1}:=J_{k-1}-J_k$ and $E:=\underset{k\in[N]}{\cup}E_{k-1}$, then $\mu(E_{k-1})<2^{-k}\epsilon$ and $\mu(E)<\epsilon$. 

Let $N\in\mathbb{N}$ be large enough so that $\eta^N<\epsilon$. Note that for $\forall k\in[N]$, 
$\mu(\Para\cap J_{k-1})<\mu(E_{k-1})+\mu(\Para\cap J_{k})$.

Adding these inequalities over all $k\in[N]$ gives,
\[\mu(\Para)=\mu(\Para\cap J_0)< \underset{k\in[N]}{\Sigma}\mu(E_{k-1}) + \mu\left( \cap J_N\right) < \epsilon + \mu\left(\left[\underset{k\in[N]}{\cap}\underset{i\in M_k}{\cap}P_{n_i}\right] \cap J_N\right).\]

Now note that by inductive construction, for $\forall N'\in[N]$
\[ \mu\left(\left[\underset{k\in[N']}{\cap}\underset{i\in M_k}{\cap}P_{n_i}\right] \cap J_N\right) < \eta \mu\left(\left[\underset{k\in[N'-1]}{\cap}\underset{i\in M_k}{\cap}P_{n_i}\right] \cap J_N\right) \]
Therefore, $\mu\left(\left[\underset{k\in[N']}{\cap}\underset{i\in M_k}{\cap}P_{n_i}\right] \cap J_N\right) < \eta^N\mu(J_0)<\epsilon$ and $\mu(\Para)<2\epsilon$. Thus Theorem \myref{thm:A} is proved. \qed

\section{Proof of Proposition \ref{lem:unif_bound_1}}\label{sect:uni_bound}

To prove Proposition \myref{lem:unif_bound_1}, we will use the following two lemmas on the space of parameterized circle diffeomorphisms with t-winding number 1, namely,
\begin{equation}\label{eqn:wind_one}
f_t(\theta)=\theta+t+g_t(\theta)\bmod 1
\end{equation}
The first one is a result by Herman which is an improvement of Arnold's [Theorem 2, \cite{Arnold}].
Let for $\forall C>0$, \boldmath $\Dioph(C)$ \unboldmath denote the set \{$x\in[0,1]$ : $\forall n\in\num -\{0\}$, $|e^{i2\pi nx}-1|\geq C|n|^{-3}$\}. It is known that $\forall C>0$, $\Dioph(C)$ is compact and $\underset{C\to 0}{\lim}\mu \Dioph(C) = 1$.

\begin{lemma}[KAM theorem, \cite{HermanSeminal}]\label{lem:KAM_fund}
Let $f_t$ be as in Eq. \myref{eqn:wind_one} and $C>0$. Then $\exists K_0(C)>0$ (with $K_0(C)\to 0$ as $C\to 0$) and $L(C)>0$ (with $L(C)\to \infty$ as $C\to 0$) such that if the periodic part $g$ of $f$ is $C^3$ and $\|g\|_{C^3}=K\leq K_0$, then $\exists$ a continuous map $\lambda_g:\Dioph(C)\to\real $ such that for $\forall t\in \Dioph(C)$, $\exists$ a diffeomorphism $h_{g,t}$ of $S^1$ such that the following hold for $\forall t\in \Dioph(C)$.
\\(i) $\theta+\lambda_g(t)+g(\theta)\bmod 1=h_{g,t}^{-1}\circ(\theta+t)\circ h_{g,t}$, i.e., the map $\theta\mapsto \theta+\lambda_g(t)+g(\theta)\bmod 1$ is conjugate via $h_{g,t}$ to a rotation by $t$,
\\(ii) $|\lambda_g(t)-t|\leq KL(C)$,
\\(iii) $|h_{f,t}-Id|_{C^0}+|df_{f,t}-Id|_{C^0}\leq KL(C)$.
\end{lemma}

The next lemma is a generalization of a lemma by Herman [3.8.2 \cite{Herman1}] and its proof is almost a replication of the proof given in [7.1. \cite{Herman1}]. The proof has been provided here for the sake of completeness and since no equivalent lemma has been found by the author in the mathematical literature. 

\begin{lemma}[Generalization of Herman's continuity theorem]\label{lem:con_at_Id}
Let $f_t$ be as in Eq. \myref{eqn:wind_one}. Then for $\forall\epsilon>0$, $\exists\delta>0$ such that if $g:\real \times [0,1]\to\real $ is a $C^3$ function, periodic in the first coordinate, such that $\|g\|_{C^3}<\delta$, then $\mu\Para(f_{n,t})<\epsilon$.
\end{lemma}
\begin{proof} Let $\epsilon>0$ be fixed for the rest of the proof. We are interested in the family of maps $f_t:\theta\mapsto\theta+\lambda_{g_t}(t)+g_t(\theta)$. Let $C>0$ be chosen so that $\mu(\Dioph(C))>1-\epsilon$. Let $K_0=K_0(C)$ and $L=L(C)$ be as in Lemma \myref{lem:KAM_fund} and let $K<K(C)$ be chosen so that $KL<\epsilon$. Then by Lemma \myref{lem:KAM_fund}, if $\|g\|_{C^3}<K$ and $t\in \Dioph(C)$, then 
\\(i) the map $f_t:\theta\mapsto\theta+\lambda_{g_t}(t)+g_t(\theta)$ is conjugate via a diffeomorphism $h_{g_t,t}$ to a rotation by $t$.
\\(ii)$\|h_{g,t}-Id\|_{C^1}\leq \epsilon$.
\\Therefore, if $t\in \lambda_{g_t}(\Dioph(C))$, then $f_t$ is conjugate to a rotation by $\left(\lambda_{g_t}\right)^{-1}(t)$. Let \boldmath $\Dioph'(C)$ \unboldmath :=\{$t\in \Dioph(C)$ : $t\in\lambda_{g_t}(\Dioph(C))$\}. Note that $\Para(f)\cap \Dioph'(C)=\emptyset$, so the lemma will be proved if it can be shown that $\mu(\Dioph'(C))>1-6\epsilon$. 

It is proved in 7.1. \cite{Herman1}, that $\forall t\in \cap \Dioph(C)$, $\lambda_{g_t}(\Dioph(C))$ is a compact set with Lebesgue measure $\mu(\lambda_{g_t}(\Dioph(C)))>1-\epsilon$. Note that $h_{g_t}$ changes continuously with $g_t$ (for proof, see for instance, Lemma 4, \cite{Parkhe}). Therefore $\lambda_{g_t}$ changes continuously with $g_t$ and therefore with $t$.

For $\forall s\in \Dioph(C)$, let \boldmath $\phi_s$ \unboldmath $:\Dioph(C)\mapsto \real $ denote the map $t\mapsto \lambda_{g_t}(s)$. Then $\phi_s$ is Lipschitz by Claim (ii) of Lemma \myref{lem:KAM_fund} and therefore can be extended to a Lipschitz map $\phi_s:[0,1]\mapsto \real $.

Now, $\Dioph'(C)=$\{$t\in \Dioph(C)$ : $\exists s\in \Dioph(C)$ $\ni$ $t=\lambda_{g_t}(s)$\}. Let \boldmath $\Dioph''(C)$ \unboldmath $:=$ \{$s\in \Dioph(C)$ : graph of $\phi_s$ intersects the diagonal $s=t$\}. Now note that for $\forall s\in[\epsilon,1-\epsilon]$, the graph of $\phi_s$ intersects the diagonal. Let $\bar{\phi}(s)$ denote this point of intersection. Since $\|\phi_s-Id\|_{C^{0}}<\epsilon$, we have $\|\bar\phi-Id\|_{C^{0}}<\epsilon$. Therefore, 
\[\mu\left(\Dioph''(C)\right)\geq \mu\left(\bar\phi([\epsilon,1-\epsilon])\right)>(1-\epsilon)(1-2\epsilon),\] 
Note that $\Dioph'(C)=\Dioph(C)\cap \Dioph''(C)$, therefore, we have,
\[ \mu\left(\Dioph'(C)\right)=\mu\left(\Dioph''(C)\cap\Dioph(C)\right)> \mu\left(\Dioph''(C)\right)+\mu\left(\Dioph(C)\right)-1 > 1-4\epsilon+2\epsilon^2-\epsilon > 1-6\epsilon. \qed\]
\end{proof}

\textbf{Proof of the proposition.} Let \boldmath$\eta$\unboldmath $:[0,1]\to[0,1]$ be defined as \boldmath $\eta(r)$ \unboldmath $:=\sup\{\mu\Para(f)\ :\ f\in\Family,\ \|f_t\|_{\Family}<r\}$ for $\forall r\in[0,1]$.

Note that $\eta$ is a non-decreasing function of $r$. Moreover, $\eta(0)=0, \eta(1)=1$. Define the universal constant \boldmath $r_*$ \unboldmath as $r_*:=\inf\{r\in[0,1]\ :\ \eta(r)=1\}$. Lemma \myref{lem:con_at_Id} proves that $r_*>0$. Proposition 6.2. of \cite{Herman1} establishes a bound $\mu(\Para(f_t))<1-\delta$, where $\delta$ depends only on $|\partial_t g_t|$ , which is $<\|f_t\|_{\Family}$. So $r^*=1$. Thus we have proved the existence of the constant $\eta\in(0,1)$ claimed in Proposition \myref{lem:unif_bound_1}, for the case when $\wind(f_t)=1$. We will extend to the general case below.

Given $R\in(0,1)$, let $\eta=\eta(R)$. Therefore, $\eta\in(0,1)$. Let $f\in\Family$ be of the form $f_t:\theta\mapsto\theta+nt+g(\theta,t)\bmod 1$, where $|\partial_t g|<R$ and $|\partial_\theta g)|<R$. Let $t$ be renormalized over the interval $[0,N]$. We will define $n$ parameterized families $F_k$ of circle diffeomorphisms, for each $k\in[n]$.
\[F_{k,t}(\theta)=\theta +t+G_k(\theta, t)\bmod 1,\ \ \ \mbox{ where } G_k(\theta, t)= g(\theta,\frac{t+k}{n}).\]
Note that $|\partial_t G_k|<R$ and $|\partial_\theta G_k)|<R/N<R$. Therefore, for $\forall k\in[n]$, $\mu\Para(F_{k})<\eta$. Therefore, 
\[\mu\Para(f)=\frac{1}{n}\underset{k\in[n]}{\Sigma}\mu\Para(F_k)<\eta.\qed\]

This completes the proof of Proposition \myref{lem:unif_bound_1} and therefore of Theorem \myref{thm:A}. \qed

\bibliographystyle{unsrt}
\bibliography{1D_bibliography}
\end{document}